\documentclass[11pt,a4paper]{amsart}
\usepackage{amsfonts,amssymb}
\usepackage{tikz}
\usetikzlibrary{matrix,arrows,calc}
\usepackage{graphicx,mathrsfs,ifpdf}
\usepackage{url}

\newtheorem{theorem}{Theorem}
\newtheorem{lemma}[theorem]{Lemma}

\newtheorem{conjecture}{Conjecture}

\theoremstyle{definition}

\numberwithin{equation}{section}
\numberwithin{figure}{section}

\DeclareMathOperator{\diam}{diam}
\DeclareMathOperator{\dist}{dist}

\DeclareMathOperator{\RR}{\mathbb R}

\ifpdf
\usepackage[pdftex]{hyperref}
\else
\newcommand{\href}[2]{#2}
\fi

\usepackage{chngcntr}
\counterwithout{figure}{section}

\newenvironment{tikzgraph}
  {\begin{tikzpicture}
      [vertex/.style={circle, draw=black, fill, inner sep=0mm, minimum size=3pt},edge/.style={semithick},
       subdivision/.style={circle, draw=black, fill=white, inner sep=0mm, minimum size=3pt},edge/.style={semithick}]\begin{scope}}
  {\end{scope}\end{tikzpicture}}

\begin{document} %\onehalfspacing
\title[Fullerene graphs of small diameter]{Fullerene graphs of small diameter}
\author{Diego Nicodemos}
\thanks{Partially supported by CAPES and CNPq.}
\address{Col\'egio Pedro II, COPPE/Sistemas, Universidade Federal do Rio
  de Janeiro, Brazil}
\email{nicodemos@cos.ufrj.br}
\author{Mat\v ej Stehl\'ik}
\thanks{Partially supported by ANR project Stint (ANR-13-BS02-0007) and by
LabEx PERSYVAL-Lab (ANR-11-LABX-0025).}
\address{Laboratoire G-SCOP, Universit\'e Grenoble Alpes, France}
\email{matej.stehlik@grenoble-inp.fr}
\date{}

\begin{abstract}
A fullerene graph is a cubic bridgeless plane graph with only pentagonal and
hexagonal faces. We exhibit an infinite family of fullerene graphs of diameter
$\sqrt{4n/3}$, where $n$ is the number of vertices. This disproves a
conjecture of Andova and \v Skrekovski [MATCH Commun. Math. Comput. Chem.
70 (2013) 205--220], who conjectured that every fullerene graph on $n$ vertices
has diameter at least $\lfloor \sqrt{5n/3}\rfloor-1$.
\end{abstract}

\maketitle
 
\section{Introduction}

\emph{Fullerene graphs} are cubic bridgeless plane graphs with only pentagonal and
hexagonal faces. Their beautiful structure---along with the fact that they can
serve as models for fullerene molecules---have attracted many researchers, and there
is now a wide body of literature on the various properties and parameters of
fullerene graphs.

One parameter that has received relatively little attention is the \emph{diameter},
defined for a graph $G$ as the maximum distance between two vertices of $G$, and denoted
by $\diam(G)$. Andova et al.~\cite{ADKLS12} have shown that if $G$ is a fullerene graph
on $n$ vertices, then $\diam(G) \geq \sqrt{2n/3-5/18}-\frac12$. In a subsequent paper,
Andova and \v Skrekovski~\cite{AndSkr13} studied the diameter of fullerene graphs with
full icosahedral symmetry. Believing that these fullerene graphs minimise the diameter,
they proposed the following conjecture.

\begin{conjecture}[Andova and \v Skrekovski~\cite{AndSkr13}]
\label{conj:diameter}
  If $G$ is any fullerene graph on $n$ vertices, then
  $\diam(G) \geq \lfloor \sqrt{5n/3}\rfloor-1$.
\end{conjecture}

The conjecture bears a striking resemblance to a famous conjecture in
differential geometry due to Alexandrov~\cite{Ale55}, which states that
$D \geq \sqrt{2A/\pi}$, for any closed orientable surface of area $A$ and intrinsic
diameter $D$. The bound in Alexandrov's conjecture is attained by the
doubly-covered disk, a degenerate surface formed by gluing two discs along
their boundaries. 

By a deep theorem of Alexandrov (see e.g.~\cite[Theorem 23.3.1]{DemORo07}
or~\cite[Theorem 37.1]{Pak10}), any fullerene graph can be embedded
in the surface $\partial P$ of a convex (possibly degenerate) polyhedron
$P \subset \RR^3$ so
that every face is isometric to a regular pentagon or a regular hexagon
with unit edge length, and this polyhedron is unique up to isometry of
$\RR^3$. (We should stress that the edges of the polyhedron may not
correspond to the edges of the graph.)
This allows us to view fullerene graphs as geometric objects, and to talk
about the `shape' of a fullerene graph.

The fullerene graphs with full icosahedral symmetry investigated by Andova and
\v Skrekovski~\cite{AndSkr13} have a rather `spherical' shape. However,
since the minimisers in Alexandrov's conjecture are doubly-covered discs,
it seems that fullerene graphs which minimise the diameter,
for a given number of vertices, ought to resemble a disc.
This led us to study the class of fullerene graphs which were called \emph{nanodiscs} by
Graver and Monacino~\cite{GraMon08}. We were able to show that they have diameter at most
$\sqrt{4n/3}$, thus disproving Conjecture~\ref{conj:diameter}. (The smallest counterexample
we found has $300$ vertices.)

\begin{theorem}
\label{thm:main}
  For every $r \geq 2$ and every $1 \leq t \leq r-1$, there exists a fullerene graph
  $D_{r,t}$ on $12r^2$ vertices of diameter at most $4r$. In particular,
  $\diam(D_{r,t})\leq\sqrt{4n/3}$.
\end{theorem}

The graph $D_{r,t}$ is best defined using the planar dual graph.
Let $T$ be the (infinite) $6$-regular planar triangulation. Fix a vertex $u \in V(T)$,
and let $T_r(u)$ be the subgraph of $T$ induced by all the vertices at distance at
most $r$ from $u$. Let $C=(u_1,u_2, \ldots, u_{6r})$ be the outer cycle of $T_r(u)$;
clearly, $C$ has six vertices of degree $3$ (say $u_{kr}$, for $1\leq k \leq 6$),
and the other vertices $u_i$ all have degree $4$.

Let $r$ and $t$ be integers such that $r \geq 2$ and $1 \leq t \leq r-1$. Take two
copies $T_r(n)$ and $T_r(s)$ of the graph defined above (one with centre $n$ and the
other with centre $s$), and suppose $u_1, \ldots, u_{6r}$ is the outer cycle of
$T_r(n)$ and $v_1, \ldots, v_{6r}$ is the outer cycle of $T_r(s)$. The graph $D_{r,t}^*$
is obtained from the disjoint union of $T_r(n)$ and $T_r(s)$ by identifying the vertex
$u_i$ and $v_{i+t}$, for all $1 \leq i \leq 6r$ (the indices are taken modulo $6r$).

Clearly, $D_{r,t}^*$ is a planar triangulation, with all vertices of degree $5$ and $6$,
so the planar dual $D_{r,t}$ is a fullerene graph. Let us remark that the automorphism
group of $D_{r,t}$ is $D_6$, unless $r=2t$ in which case the automorphism group is $D_{6d}$.
The graphs $D^*_{2,1}$ and $D_{2,1}$ are shown in Figure~\ref{fig:D21}.

\begin{figure}
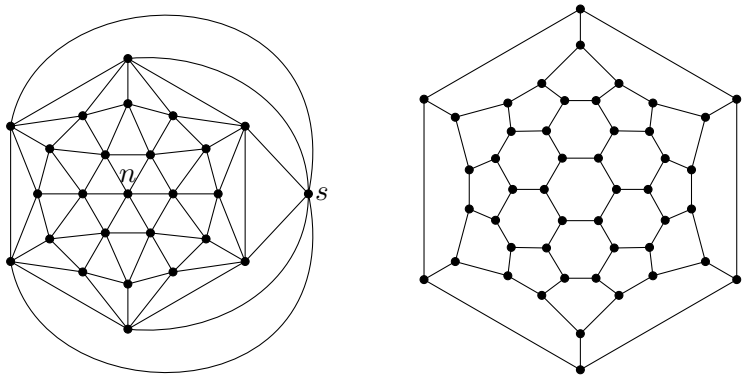

\centering
\begin{tikzgraph}[scale=0.6]
\draw[draw=none, use as bounding box](-3,-4) rectangle (4,4);
\foreach\i in {0,...,5}
{
  \path (60*\i:1) coordinate (a\i);
  \path (60*\i+30:3) coordinate (c\i);
}
\foreach\i in {0,...,11}
{
  \path (30*\i:2) coordinate (b\i);
}

\draw (a0)--(a1)--(a2)--(a3)--(a4)--(a5)--cycle
      (b0)--(b1)--(b2)--(b3)--(b4)--(b5)--(b6)--(b7)--(b8)--(b9)--(b10)--(b11)--cycle
      (c0)--(c1)--(c2)--(c3)--(c4)--(c5)--cycle
      (0,0)--(a0)
      (0,0)--(a1)
      (0,0)--(a2)
      (0,0)--(a3)
      (0,0)--(a4)
      (0,0)--(a5)
      (a0)--(b0)
      (a1)--(b2)
      (a2)--(b4)
      (a3)--(b6)
      (a4)--(b8)
      (a5)--(b10)
      (a0)--(b1)--(a1)--(b3)--(a2)--(b5)--(a3)--(b7)--(a4)--(b9)--(a5)--(b11)--cycle
      (b1)--(c0)
      (b3)--(c1)
      (b5)--(c2)
      (b7)--(c3)
      (b9)--(c4)
      (b11)--(c5)
      (b0)--(c0)--(b2)--(c1)--(b4)--(c2)--(b6)--(c3)--(b8)--(c4)--(b10)--(c5)--cycle
      (4,0)--(c0)
      (4,0) .. controls (4,2) and (2,3.2) .. (c1)
      (4,0) .. controls (5,5) and (-2,5) .. (c2)
      (4,0) .. controls (5,-5) and (-2,-5) .. (c3)
      (4,0) .. controls (4,-2) and (2,-3.2) .. (c4)
      (4,0)--(c5);

\draw (0,0) node[vertex] {};
\draw (4,0) node[vertex] {};
\foreach\i in {0,...,5}
{
  \draw (a\i) node[vertex] {};
  \draw (c\i) node[vertex] {};
}
\foreach\i in {0,...,11}
{
  \draw (b\i) node[vertex] {};
}
\draw (0,0.4) node {$n$};
\draw (4.3,0) node {$s$};
\end{tikzgraph}
\hfil
\begin{tikzgraph}[scale=0.6]
\foreach\i in {0,...,5}
{
  \path (60*\i:0.8) coordinate (a\i);
  \path (60*\i+30:4) coordinate (d\i);
}
\foreach\i in {1,4,...,16}
{
  \path (20*\i:2) coordinate (b\i);
}
\foreach\i in {2,5,...,17}
{
  \path (20*\i:2) coordinate (b\i);
}
\foreach\i in {1,4,...,16}
{
  \path (20*\i+10:3.2) coordinate (c\i);
}
\foreach\i in {0,3,...,15}
{
  \path (20*\i:1.5) coordinate (b\i);
}
\foreach\i in {0,3,...,15}
{
  \path (20*\i+10:2.5) coordinate (c\i);
}
\foreach\i in {2,5,...,17}
{
  \path (20*\i+10:2.5) coordinate (c\i);
}

\draw (a0)--(a1)--(a2)--(a3)--(a4)--(a5)--cycle
      (b0)--(b1)--(b2)--(b3)--(b4)--(b5)--(b6)--(b7)--(b8)--(b9)--(b10)--(b11)--(b12)--(b13)--(b14)--(b15)--(b16)--(b17)--cycle
      (c0)--(c1)--(c2)--(c3)--(c4)--(c5)--(c6)--(c7)--(c8)--(c9)--(c10)--(c11)--(c12)--(c13)--(c14)--(c15)--(c16)--(c17)--cycle
      (d0)--(d1)--(d2)--(d3)--(d4)--(d5)--cycle
      (a0)--(b0)
      (a1)--(b3)
      (a2)--(b6)
      (a3)--(b9)
      (a4)--(b12)
      (a5)--(b15)
      (b1)--(c0)
      (b2)--(c2)
      (b4)--(c3)
      (b5)--(c5)
      (b7)--(c6)
      (b8)--(c8)
      (b10)--(c9)
      (b11)--(c11)
      (b13)--(c12)
      (b14)--(c14)
      (b16)--(c15)
      (b17)--(c17)
      (d0)--(c1)
      (d1)--(c4)
      (d2)--(c7)
      (d3)--(c10)
      (d4)--(c13)
      (d5)--(c16)
;

\foreach\i in {0,...,5}
{
  \draw (a\i) node[vertex] {};
  \draw (d\i) node[vertex] {};
}
\foreach\i in {0,...,17}
{
  \draw (b\i) node[vertex] {};
  \draw (c\i) node[vertex] {};
}
\end{tikzgraph}
\caption{The plane triangulation $D^*_{2,1}$ and its dual fullerene graph $D_{2,1}$.}
\label{fig:D21}
\end{figure}

\section{The proof}

Our graph-theoretic terminology is standard and follows~\cite{BonMur08}.
To prove Theorem~\ref{thm:main}, we will make use of the following simple lemma,
which relates distances in a fullerene graph to distances in its dual graph.

\begin{lemma}
\label{lem:diameter-dual}
  Let $G$ be a fullerene graph and $G^*$ its dual graph. Fix any pair of vertices
  $A,B \in V(G)$, and let $u$ and $v$ be faces of $G$ incident to $A$ and $B$,
  respectively. Then $\dist_G(A,B) \leq 2\dist_{G^*}(u,v)+3$.
\end{lemma}

\begin{proof}
  Let $k=\dist_{G^*}(u,v)$, and let $P^*$ be a path of length $k$ between $u$ and
  $v$ in the dual graph $G^*$. So $P^*=u_0,u_1,\ldots, u_k$, where $u_0=u$ and
  $u_k=v$. Let $\delta(P^*)$ be the edges with precisely one end vertex in $V(P^*)$.
  Since the vertices in $G^*$ have degree at most $6$, $|\delta(P^*)| \leq 6+4k$.
  Therefore the dual edges to the cut $\delta(P^*)$ form a cycle in $G$ of length
  at most $6+4k$ containing the vertices $A$ and $B$. Therefore, $\dist_G(A,B) \leq 2k+3$.
\end{proof}

Recall that the graph $D^*_{r,t}$ has two special vertices $n$ and $s$, which can be
thought of as the north and south poles. To continue the analogy with geography, we
shall define the \emph{latitude} $\varphi(u)$ of a vertex $u \in V(D^*_{r,t})$ as
\[
\varphi(u)=
\begin{cases}
  r-\dist(n,u)    & \text{ if } \dist(n,u)\leq r\\
  -r+\dist(s,u) & \text{ if } \dist(s,u)\leq r.  
\end{cases}
\]
In particular, $n$ has latitude $r$, $s$ has latitude $-r$, and vertices at distance $r$ from $n$
and $s$ have latitude $0$, i.e., they lie on the `equator'.

\begin{proof}[Proof of Theorem~\ref{thm:main}]
  Fix two vertices $A, B \in V(D_{r,t})$. For convenience of notation we will consider the dual
  graph $D^*_{r,t}$ as a pure $2$-dimensional simplicial complex. The vertices $A,B$ correspond
  to faces $A=\{u_1,u_2,u_3\}$ and $B=\{v_1,v_2,v_3\}$ in the dual triangulation $D^*_{r,t}$.
  Note that every face of $D^*_{r,t}$ is incident to two vertices at the same latitude, and
  another vertex at a different latitude.
  Assume without loss of generality that $\varphi(u_1)\neq \varphi(u_2)$, and
  $\varphi(v_1)\neq \varphi(v_2)=\varphi(v_3)$. Furthermore, let $B'=\{v_1',v_2,v_3\}$ be
  the (unique) face incident to $\{v_2,v_3\}$ which is different from $B$.

  There exists a path $P_u$ of length $2r$ from $n$ to $s$ containing the subpath $u_1,u_2$,
  and a path $P_v$ of length $2r$ from $n$ to $s$ containing the subpath $v_1,v_2,v_1'$.
  (To see this, it is enough to note that every vertex $v \in V(D^*_{r,t})\setminus\{n,s\}$ has
  a `northern' neighbour $v_N$ of latitude $\varphi(v_N)=\varphi(v)+1$ and a `southern' neighbour
  $v_S$ of latitude $\varphi(v_S)=\varphi(v)-1$.)
  The union $P_u \cup P_v$ is a closed walk $W$ of length $4r$, which must contain two subwalks
  $W_1,W_2$ from $\{u_1,u_2\}$ to $\{v_1,v_1'\}$, of length $\ell_1,\ell_2$, respectively.
  Without loss of generality assume that $\ell_1 \leq \ell_2$. Since $\ell_1+\ell_2 \leq 4r-3$,
  it follows that $\dist_{D^*_{r,t}}(\{u_1,u_2\},\{v_1,v_1'\}) \leq \ell_1 \leq \lfloor\frac12(4r-3)\rfloor=2r-2$.
  
  If $\dist_{D^*_{r,t}}(\{u_1,u_2\},v_1) \leq 2r-2$, then $\dist_{D_{r,t}}(A,B) \leq 4r-1$
  by Lemma~\ref{lem:diameter-dual}.
  If $\dist_{D^*_{r,t}}(\{u_1,u_2\},v_1')\leq 2r-2$, then $\dist_{D_{r,t}}(A,B') \leq 4r-1$
  by Lemma~\ref{lem:diameter-dual} again, but then clearly $\dist_{D_{r,t}}(A,B) \leq 4r$.
  This completes the proof.
\end{proof}

\bibliographystyle{plainurl}

\end{document}